\documentclass[suppldata]{interact}

\usepackage{epstopdf} 
\usepackage{amsmath}
\usepackage[caption=false]{subfig} 
\usepackage{amsfonts}

\usepackage[longnamesfirst,sort]{natbib} 
\bibpunct[, ]{(}{)}{;}{a}{,}{,} 
\renewcommand\bibfont{\fontsize{10}{12}\selectfont}

\usepackage{hyperref}
\usepackage{bbm}

\theoremstyle{plain}
\newtheorem{theorem}{Theorem}[section]
\newtheorem{lemma}[theorem]{Lemma}
\newtheorem{corollary}[theorem]{Corollary}
\newtheorem{proposition}[theorem]{Proposition}

\newtheorem{mydef*}{Definition*}

\theoremstyle{definition}
\newtheorem{definition}[theorem]{Definition}

\theoremstyle{remark}

\begin{document}
\title{FEW REMARKS ON ESSENTIAL SUBMODULES}
\author{
\name{Sourav Koner and Biswajit Mitra}\thanks{CONTACT Email:  harakrishnaranusourav@gmail.com \& bmitra@math.buruniv.ac.in}
\affil{Department of Mathematics, The University of Burdwan, India 713104 \& Department of Mathematics, The University of Burdwan, India 713104}}

\maketitle 

\begin{abstract}
In the present paper, modules over integral domains and principal ideal domains that are proper essential extensions of some submodules are classified. We introduce a new class of modules that we call $\mathrm{SM}$ modules and show that the class of Artinian modules, locally finite modules, and modules of finite lengths are all proper subclasses of SM modules. We also show that non-semisimple $\mathrm{SM}$ modules possess essential socles. Further, we show that non-semieimple modules over integral domains with nonempty torsion-free parts do not possess essential socles.
\end{abstract}



\section{Introduction}

An essential submodule of a module is a submodule that intersects any other nonzero submodule nontrivially. It is named by Eckman and Schpof (see [\cite{ES53}]) though the idea behind essential submodule in a left $R$-module $M$ is due to Johnson (see [\cite{RJ51}]). Essential submodules play a crucial role in noncommutative algebra. As an example, a necessary and sufficient condition that \emph{a left ideal in a ring $R$ is essential, is the existence of a nonzero-divisor (that is, a regular element)}, is again a necessary and sufficient condition for a ring $R$ to have a left classical Artin semisimple quotient ring, that is, $R$ is semiprime left Goldie ring. This shows that the self-generating property of essential submodules generate other essential objects. In much of the literature, the definition of an essential submodule includes the parent module. In this paper, we work with proper essential submodules whose definition is the following:

\begin{definition}\label{df}
A proper nontrivial submodule $E$ of a left $R$-module $M$ is said to be  a proper essential submodule of $M$ if  $E \cap N \neq \{0_{M}\}$ for every nontrivial submodule $N$ of $M$. 
\end{definition} 

Throughout the paper, unless otherwise stated, we assume $R$ is a ring with unity $1 \neq 0$ which is not a field. Observe that any non-trivial proper ideal $I$ in an integral domain $R$ is a proper essential submodule of the $R$-module $R$. However, it is not necessary for an $R$-module to contain a proper essential submodule. Because, if we consider a finite $\mathbb{Z}$-module $M$ of order $n$, where $n$ is square-free, and if $N_{p}$ denote the submodule of $M$ of order $p$ where $p$ is a prime dividing $n$, then $M = \bigoplus_{p \mid n} N_{p}$ shows that $M$ does not contain any proper essential submodule. The following proposition is our observation and is supposed to prevail in the literature, but we could not trace the result either in direct or indirect form anywhere. So we proved this result and incorporated it here. 

\begin{proposition}\label{history}
If an $R$-module $M$ is semisimple, then $M$ does not contain any proper essential submodule. Hence, any $\mathbbm{k}$-vector space $V$ does not contain any proper essential subspace.
\end{proposition}
\begin{proof}
Let $M = \bigoplus_{\alpha \in \Lambda} N_{\alpha}$, where each $N_{\alpha}$ is a simple submodule of $M$. If a submodule $E$ of $M$ intersects every other non-trivial submodule of $M$ non-trivially, then $N_{\alpha} \subseteq E$ for all $\alpha \in \Lambda$. But this implies $\bigoplus_{\alpha \in \Lambda} N_{\alpha} \subseteq E$, that is, $M \subseteq E$. 

Let $\{v_{\alpha}\}_{\alpha \in \Lambda}$ be a basis for the $\mathbbm{k}$-vector space $V$. Then we have $V =\bigoplus_{\alpha \in \Lambda} V_{\alpha}$, where $V_{\alpha}$ is the subspace generated by $v_{\alpha}$. As each $V_{\alpha}$ is a simple $\mathbbm{k}$-subspace of $V$, we can use the first part of this proposition to conclude that $V$ does not contain any proper essential subspace.
\end{proof}

A part of the following theorem \eqref{baba} is very much well-known (see [\cite{TL99}]). Therefore, we include proof of the other part of the theorem which we again did not find in the literature.

\begin{theorem}\label{baba}
Let $M$ be a left module over $R$. Then one can show the equivalence
 
$(\mathrm{a})$ Every submodule of $M$ is a direct summand.

$(\mathrm{b})$ $M$ has no proper essential submodule.

$(\mathrm{c})$ $M$ is semisimple.

$(\mathrm{d})$ $\mathrm{Soc}(M) = M$.
\end{theorem}

\begin{proof}
Equivalency of (a), (c) and (d) are found in [\cite{TL99}]. Here, we show the equivalence of (a) and (b), that is, $\mathrm{(a)} \Leftrightarrow \mathrm{(b)}$.

$\mathrm{(a)} \Rightarrow \mathrm{(b)}:$ Suppose $M$ has a proper essential submodule. Then from the definition, it trivially follows that the proper essential submodule can not be the direct summand. 

$\mathrm{(b)} \Rightarrow \mathrm{(a)}:$ Conversely let $N$ be a proper submodule of $M$. Then $N$ is not proper essential. There exists a submodule $T$ such that $N\cap T = \{0\}$. Choose a maximal such $T$, then $N\bigoplus T$ is either proper essential in $M$ or $N \bigoplus T = M$. But as per hypothesis, $N \bigoplus T = M$, so the result follows. 
\end{proof}

Theorem \eqref{babama} can be found in [\cite{JL66}]), so, we skip its proof.
 
\begin{theorem}\label{babama}
The intersection of all essential submodules of a left $R$-module $M$ is equal to its socle $\mathrm{Soc}(M)$.
\end{theorem}

\section{Proper essential submodules over integral and principal ideal domains}

\begin{theorem}\label{maa}
Let $\{M_{\omega}\}_{\omega \in \Lambda}$ be a family of left $R$-modules indexed by a nonempty set $\Lambda$. Then $\bigoplus_{\omega \in \Lambda} M_{\omega}$ has a proper essential submodule if and only if some $M_{\omega}$ has a proper essential submodule. 
\end{theorem}
\begin{proof}
Let $M = \bigoplus_{\omega \in \Lambda} M_{\omega}$ and $\omega_{0} \in \Lambda$ be such that $M_{\omega_{0}}$ has a proper essential submodule $E_{\omega_{0}}$. Let $X = \bigoplus_{\omega \in \Lambda} X_{\omega}$, where $X_{\omega_{0}} = E_{\omega_{0}}$ and $X_{\omega} = M_{\omega}$ if $\omega \neq \omega_{0}$. We claim that $X$ is a proper essential submodule of $M$. Let $N$ be any nontrivial submodule of $M$ and let $(n_{\omega})_{\omega \in \Lambda}$ be a nonzero element of $N$. Clearly, $(n_{\omega})_{\omega \in \Lambda} \in X$ if $n_{\omega_{0}} = 0_{M_{\omega_{0}}}$. If $n_{\omega_{0}} \neq 0_{M_{\omega_{0}}}$, consider the cyclic module $\langle n_{\omega_{0}} \rangle$. As $E_{\omega_{0}} \cap \langle n_{\omega_{0}} \rangle \neq \{0_{M_{\omega_{0}}}\}$, we conclude that $X \cap N \neq \{0_{M}\}$. 

For the converse, assume that $M$ has a proper essential submodule $E$ but $M_{\omega}$ does not have a proper essential submodule for all $\omega \in \Lambda$. Observe that, for all $\omega \in \Lambda$, we have either $E \cap M_{\omega}$ is proper essential in $M_{\omega}$ or $E \cap M_{\omega} = M_{\omega}$. As $M_{\omega}$ does not have any proper essential submodule, it must be that $E \cap M_{\omega} = M_{\omega}$ for all $\omega \in \Lambda$, that is, $M \subseteq E$, a contradiction. 
\end{proof}

Free modules play a central role in algebra, since any module is the homomorphic image of some free module. The following corollary is now immediate from theorem \eqref{maa}.

\begin{theorem}\label{spm} 
Let $R$ be an integral domain and let $A$ be a nonempty set. Then the free $R$-module $R(A)$ on the nonempty subset $A$ has a proper essential submodule.
\end{theorem}
\begin{proof}
Since $R(A)$ is free on the subset $A$, we have $R(A) \simeq \bigoplus_{A} R$. The proof is now completed using theorem \eqref{maa}.
\end{proof}

An immediate consequence of theorem \eqref{maa} is the classification of finitely generated modules over principal ideal domains possessing proper essential submodules. Before that, we look the following lemma.

\begin{lemma}\label{ons}
The non-simple left cyclic module $\langle m \rangle$ over an integral domain $R$ has a proper essential submodule if and only if there exists a maximal ideal $\mathcal{M}$ of $R$ containing $\mathrm{Ann}_{R}(m)$ such that for any ideal $I$ of $R$ containing $\mathrm{Ann}_{R}(m)$ we have $\alpha \beta \notin \mathrm{Ann}_{R}(m)$ for some $\alpha \in \mathcal{M} \setminus \mathrm{Ann}_{R}(m)$ and for some $\beta \in I \setminus \mathrm{Ann}_{R}(m)$. Hence, the non-simple cyclic module $\langle m \rangle$ over a principal ideal domain $R$ has a proper essential submodule if and only if the generator of the $\mathrm{Ann}_{R}(m)$ is not square-free.
\end{lemma}
\begin{proof}
Let $\mathcal{M}$ be a maximal ideal of $R$ that contains $\mathrm{Ann}_{R}(m)$ and $I$ be any ideal of $R$ containing $\mathrm{Ann}_{R}(m)$. If for some $\alpha \in \mathcal{M} \setminus \mathrm{Ann}_{R}(m)$ and for some $\beta \in I \setminus \mathrm{Ann}_{R}(m)$ we have $\alpha \beta \notin \mathrm{Ann}_{R}(m)$, then we see that $\mathcal{M}/\mathrm{Ann}_{R}(m)$ is proper essential in $R/\mathrm{Ann}_{R}(m)$. As $R/\mathrm{Ann}_{R}(m) \simeq \langle m \rangle$, we conclude that $\langle m \rangle$ has a proper essential submodule.

Conversely, suppose $\langle m \rangle$ has a proper essential submodule $E$. Since we have $\langle m \rangle \simeq R/\mathrm{Ann}_{R}(m)$, there exists some ideal $J$ of $R$ containing $\mathrm{Ann}_{R}(m)$ such that $J/\mathrm{Ann}_{R}(m) \simeq E$. Put $J$ in some maximal ideal $\mathcal{M}$. As $J/\mathrm{Ann}_{R}(m)$ is proper essential in $R/\mathrm{Ann}_{R}(m)$, for any ideal $I$ of $R$ containing $\mathrm{Ann}_{R}(m)$, we have $J/\mathrm{Ann}_{R}(m) \cap I/\mathrm{Ann}_{R}(m) \neq \{0_{\langle m \rangle}\}$. As $J$ is contained in $\mathcal{M}$, we get that $\mathcal{M}/\mathrm{Ann}_{R}(m) \cap I/\mathrm{Ann}_{R}(m) \neq \{0_{\langle m \rangle}\}$.

Let $\mathrm{Ann}_{R}(m) = \langle a \rangle$ and let $a = up_{1}p_{2} \cdots p_{n}$ where $p_{i}$'s are distinct primes and $u$ is a unit. The only maximal ideals that contain $\langle a \rangle$ are the ideals $\langle p_{i} \rangle$ where $1 \leq i \leq n$. If $\alpha \in \langle p_{i} \rangle \setminus \langle a \rangle$ and $\beta \in \langle q_{i} \rangle \setminus \langle a \rangle$ where $q_{i} = p_{1}p_{2} \cdots p_{i - 1}p_{i + 1} \cdots p_{n}$, we see that $\alpha \beta \in \langle a \rangle$. Now, the first part of this lemma shows that $\langle m \rangle$ does not have a proper essential submodule.

Assume now that $a = up_{1}^{\alpha_{1}}p_{2}^{\alpha_{2}} \cdots p_{n}^{\alpha_{n}}$, where $p_{i}$'s are distinct primes, $u$ is a unit and $\alpha_{i} > 1$ for some $i$. It is now evident that the maximal ideal $\langle p_{i} \rangle$ of $R$ satisfies the property of the first of this lemma.
\end{proof}

\begin{theorem}\label{jg}
A finitely generated left module $M$ over a principal ideal domain $R$ has a proper essential submodule if and only if the betti number of $M$ is at least $1$ or some invariant factor is not square free. 
\end{theorem}
\begin{proof}
The fundamental theorem of finitely-generated modules over principal ideal domain yields the invariant-factor decomposition $M = R^{n} \bigoplus R/\langle a_{{1}} \rangle \bigoplus \cdots \bigoplus R/\langle a_{{m}} \rangle$. The proof is now completed using theorems \eqref{maa}, \eqref{spm}, and lemma \eqref{ons}  
\end{proof}

\begin{lemma}\label{ram}
A left $R$-module $M$ has a proper essential submodule if and only if the left $R/\mathrm{Ann}_{R}(M)$-module $M$ has a proper essential submodule. Hence, if an $R$-module $M$ has a proper essential submodule, then $\mathrm{Ann}_{R}(M)$ can not be a maximal ideal.
\end{lemma}
\begin{proof}
At first, assume that $R$-module $M$ has a proper essential submodule $E$. Now, let $N$ be any submodule of $R/\mathrm{Ann}_{R}(M)$-module $M$. Consider the cyclic $R/\mathrm{Ann}_{R}(M)$-module $\langle n \rangle$, where $n \in N$ be any element. As, $\langle n \rangle \cap E \neq \{0_{M}\}$, there exists $r \in R$ such that $rn \neq 0$ and $rn \in \langle n \rangle \cap E$. Then $(r + \mathrm{Ann}_{R}(M))n \in \langle n \rangle \cap E$ (as $R/\mathrm{Ann}_{R}(M)$-module). Hence, this shows that $R/\mathrm{Ann}_{R}(M)$-module $M$ has a proper essential submodule $E$. 

Conversely, assume that $R/\mathrm{Ann}_{R}(M)$-module $M$ has a proper essential submodule $E$. Let $N$ be any submodule of $M$. Then, $\langle n \rangle \cap E \neq \{0_{M}\}$ for $n \in N$ with $n \neq 0_{M}$. This implies, for some $(r + \mathrm{Ann}_{R}(M)) \in R/\mathrm{Ann}_{R}(M)$ with $r + \mathrm{Ann}_{R}(M) \neq \mathrm{Ann}_{R}(M)$, $(r + \mathrm{Ann}_{R}(M))n \in E$, that is, $rn \in E$. Thus, $E$ is proper essential in $M$.

Now, applying proposition \eqref{history} and the first part of this lemma, we get that $\mathrm{Ann}_{R}(M)$ can not be a maximal ideal.
\end{proof}

Let $M$ be a module over an integral domain $R$. If $p \in R$ is a prime, recall that the $p$-primary component $\mathrm{Ann}^{\ast}_{M}(p)$ of $M$ is the set of all elements of $M$ that are annihilated by some positive power of $p$ and its sub-primary component $\mathrm{Ann}_{M}(p^{m})$ is the set of all elements of $M$ that are annihilated by some fixed power of $p$ (i.e., annihilated by $p^{m}$).

\begin{lemma}\label{jpm}
Let $M$ be a non-faithful left module over a principal ideal domain $R$ and let $p_{1}, p_{2}, \ldots, p_{n}$ be the distinct primes that divide the generator of $\mathrm{Ann}_{R}(M)$. Then $M$ has a proper essential submodule if and only if some sub-primary component of $M$ has a proper essential submodule for some prime $p_{i}$. 
\end{lemma}
\begin{proof}
Let $\mathrm{Ann}_{R}(M) = \langle a \rangle$ and $a = up_{1}^{\alpha_{1}}p_{2}^{\alpha_{2}} \cdots p_{n}^{\alpha_{n}}$, where $p_{i}$'s are distinct primes and $u$ is a unit. For $1 \leq i \leq n$, if we set $q_{i} = p_{1}^{\alpha_{1}}p_{2}^{\alpha_{2}} \cdots p_{i - 1}^{\alpha_{i - 1}}p_{i + 1}^{\alpha_{i + 1}} \cdots p_{n}^{\alpha_{n}}$, then the homomorphism $\varphi_{i}: M \rightarrow q_{i}M$ shows that $M/\mathrm{Ker}(\varphi_{i}) \simeq q_{i}M$. Since we have $\mathrm{Ker}(\varphi_{i}) = \langle p_{i}^{\alpha_{i}} \rangle M$ and $q_{i}M = \mathrm{Ann}_{M}(p_{i}^{\alpha_{i}})$, we conclude that $M/\langle p_{i}^{\alpha_{i}} \rangle M \simeq \mathrm{Ann}_{M}(p_{i}^{\alpha_{i}})$. As the ideals $\langle p_{1}^{\alpha_{1}} \rangle$, $\langle p_{2}^{\alpha_{2}} \rangle$, \ldots, $\langle p_{n}^{\alpha_{n}} \rangle$ are pairwise comaximal, the chinese remainder theorem yields $M \simeq \bigoplus_{i = 1}^{n} \mathrm{Ann}_{M}(p_{i}^{\alpha_{i}})$. The result now follows from theorem \eqref{maa}.
\end{proof}

\begin{theorem}\label{bholenath}
Let $M$ be a left torsion module over a principal ideal domain $R$. Then $M$ has a proper essential submodule if and only if $\mathrm{Ann}_{M}(p^{i}) \subsetneq \mathrm{Ann}_{M}(p^{j})$ for some some prime $p$ and for some positive integers $i, j$ with $i < j$. 
\end{theorem}
\begin{proof}
Let $m \in M$ be such that $rm = 0$ for some $r \in R$. Let $r = up_{1}^{\alpha_{1}}p_{2}^{\alpha_{2}} \cdots p_{n}^{\alpha_{n}}$ be the prime factorisation of $r$ into primes, where $u$ is some unit. Consider the cyclic submodule $\langle m \rangle$. As the submodule $\langle m \rangle$ is annihilated by the ideal $\langle r \rangle$, we can apply lemma \eqref{jpm} to conclude that $\langle m \rangle = \bigoplus_{i = 1}^{n} \mathrm{Ann}_{\langle m \rangle} (p_{i}^{\alpha_{i}})$. Thus, $M = \sum_{{p{\text -\mathrm{prime}}}} \mathrm{Ann}^{\ast}_{M}(p)$. Now choose finite number of primes $q_{1}, q_{2}, \ldots, q_{l}$ in $R$ and consider the submodules $ \mathrm{Ann}^{\ast}_{M}(q_{1}), \mathrm{Ann}^{\ast}_{M}(q_{2}), \ldots, \mathrm{Ann}^{\ast}_{M}(q_{l})$. If $m \in \mathrm{Ann}^{\ast}_{M}(q_{j}) \bigcap \sum_{i \neq j} \mathrm{Ann}^{\ast}_{M}(q_{i})$ for some fixed $j$ with $1 \leq j \leq l$, then we have $q_{j}^{s_{j}}m = 0$ for some $s_{j} \in \mathbb{N}$ and $m = m_{1} + \cdots +m_{j - 1} + m_{j + 1}+ \cdots + m_{l}$. Also, for all $1 \leq i \leq l$ with $i \neq j$, we have $q_{i}^{s_{i}}m_{i} = 0$ for some $s_{i} \in \mathbb{N}$. As $R$ is a principal ideal domain, there exists $x, y \in R$ such that $q_{j}^{s_{j}}x + \xi y = 1$, where $\xi = q_{1}^{s_{1}}q_{2}^{s_{2}} \cdots q_{j - 1}^{s_{j - 1}}q_{j + 1}^{s_{j + 1}} \cdots q_{l}^{s_{l}}$. But this implies $m = q_{j}^{s_{j}}xm + \xi y(m_{1} + \cdots +m_{j - 1} + m_{j + 1}+ \cdots + m_{l})$, that is, $m = 0$. Hence, we have that $M = \bigoplus_{{p{\text -\mathrm{prime}}}} \mathrm{Ann}^{\ast}_{M}(p)$. Now, if $M$ has a proper essential submodule $E$, then for some prime $p$, $E \cap \mathrm{Ann}^{\ast}_{M}(p)$ is proper essential in $\mathrm{Ann}^{\ast}_{M}(p)$ (see theorem \eqref{maa}). If $\mathrm{Ann}_{M}(p^{i}) = \mathrm{Ann}_{M}(p^{j})$ for all positive integers $i, j$, then observe that $\mathrm{Ann}_{R}(\mathrm{Ann}_{M}(p)) = \langle p \rangle$. Thus, using lemma \eqref{ram}, we conclude  $R/\langle p \rangle$-vector space $\mathrm{Ann}_{M}(p)$ has a proper essential subspace $E \cap \mathrm{Ann}_{M}(p)$. Again, applying the second part of lemma \eqref{ram}, we get a contradiction.

Conversely, assume that $\mathrm{Ann}_{M}(p^{i}) \subsetneq \mathrm{Ann}_{M}(p^{j})$ for some prime $p$ and for some positive integers $i, j$ with $i < j$. We claim that $\mathrm{Ann}_{M}(p^{i})$ is proper essential in $\mathrm{Ann}_{M}(p)$. Let $N$ be any submodule of $\mathrm{Ann}_{M}(p)$ and let $n \in N$. If $n \notin \mathrm{Ann}_{M}(p^{i})$, then suppose $l$ be the smallest positive integer such that $n \in \mathrm{Ann}_{M}(p^{l})$. Observe now that $p^{l - i}n \in \mathrm{Ann}_{M}(p^{i})$ and $p^{l - i}n \neq 0$. This implies $p^{l - i}n \in N \cap \mathrm{Ann}_{M}(p^{i})$. Now apply theorem \eqref{maa} to conclude that $M$ has a proper essential submodule.
\end{proof}

If $R$ is an integral domain, then a torsion-free left $R$-module $M$ contains a copy of $R$ as a left $R$-module. Also, the $R$-module $D^{-1}R$ has a proper essential submodule. Therefore, it is natural to consider a torsion-free $R$-module $M$ that contains $D^{-1}R$ as its submodule. In this regard, the question that arises is whether $M$ contains a proper essential submodule. The following theorem answers this question.

\begin{theorem}\label{jpmkj}
Let $M$ be a left torsion-free module over an integral domain $R$ and let $D$ denotes the set of all nonzero elements in $R$. If $M$ contains $D^{-1}R$ as its submodule, then the $R$-module $M$ has a proper essential submodule.
\end{theorem}
\begin{proof}
Let $\frac{r}{d} \in D^{-1}R$ be such that $\frac{r}{d} \neq 0_{M}$. Now, consider the collection  $\mathcal{C} = \{N \leq M \mid  R \subseteq N, \frac{r}{d} \notin N\}$ of submodules of $M$. Clearly $\mathcal{C}$ is a partially ordered set under inclusion. By Zorn's Lemma, it has a maximal element $\mathbbm{E}$. We claim that $\mathbbm{E}$ is a proper essential submodule of $M$. If not, there exists a non-trivial submodule $K$ that intersects $\mathbbm{E}$ trivially, that is, $\mathbbm{E} \cap K = \{0_{M}\}$. Observe that $K$ must not contain any nonzero element in $D^{-1}R$, because, if $\frac{r^{\prime}}{d^{\prime}} \in D^{-1}R$ is such that $\frac{r^{\prime}}{d^{\prime}} \neq 0_{M}$ and $\frac{r^{\prime}}{d^{\prime}} \in K$, then $r^{\prime} \in K$, contrary to our assumption that $\mathbbm{E} \cap K = \{0_{M}\}$. Let $\alpha \in K$ be such that $\alpha \neq 0_{M}$. Clearly, we have $\mathbbm{E} \cap \langle \alpha \rangle = \{0_{M}\}$. Consider the submodule $\mathbbm{E} \oplus \langle \alpha \rangle$. By maximality of $\mathbb{E}$, it must be that $\frac{r}{d} \in \mathbbm{E} \oplus \langle \alpha \rangle$. This implies $r/d + \alpha \in \mathbbm{E} \oplus \langle \alpha \rangle$. If $\frac{r}{d} + \alpha \in \mathbbm{E}$, then $r + d\alpha \in \mathbbm{E}$. But this gives, $d\alpha \in \mathbbm{E}$, a contradiction. Also, if $\frac{r}{d} + \alpha \in \langle \alpha \rangle$, then $r \in \langle \alpha \rangle$, a contradiction. Thus, we can write $\frac{r}{d} + \alpha = e + r^{\ast}\alpha$, where $e \in \mathbbm{E} \setminus \{0_{M}\}$ and $r^{\ast} \in R \setminus \{0_{R}\}$. But this implies $r + d\alpha = de + dr^{\ast}\alpha$. Hence, we get that $r - de = d(r^{\ast} - 1)\alpha$. Observe that $r - de \in \mathbbm{E}$ whereas $d(r^{\ast} - 1)\alpha \in \langle \alpha \rangle$. Hence, it must be that $r - de = 0_{M}$ and $d(r^{\ast} - 1) = 0_{R}$. As $d \neq 0_{R}$, it must be that $r^{\ast} = 1$. Plugging the value of $r^{\ast}$ in the equation $\frac{r}{d} + \alpha = e + r^{\ast}\alpha$ yields $\frac{r}{d} = e$. But this is absurd because $\mathbbm{E}$ being a member of $\mathcal{C}$ cannot contain $\frac{r}{d}$.  
\end{proof}

Another question that arises naturally from the definition of proper essential submodule is: if $M$ has a proper essential submodule, does the quotient module $M/N$ also have a proper essential submodule where $N$ is any proper non-trivial submodule of $M$? In general, the answer is no (example: $\mathbb{Z}$-module $\mathbb{Z}_{p^{2}}$). But under certain constraints, we can give an affirmative answer. The converse, however, always holds.

\begin{lemma}\label{gmkj}
Let $M$ be a left $R$-module and $N$ be a submodule of $M$ such that $M/N$ has a proper essential submodule $E/N$. Then, the submodule $E$ is a proper essential submodule of $M$. Moreover, if a module $M$ has a proper essential submodule $E$ and if $N$ is a submodule of $M$ such that $M/N$ is torsion-free and $N \neq E$, then the submodule $(E + N)/N$ is proper essential in $M/N$.
\end{lemma}
\begin{proof}
If $N = \{0_{M}\}$, there is nothing to show. Assume $N \neq \{0_{M}\}$. Let $L$ be a nontrivial proper submodule of $M$ and $u \in (L + N)/N \cap E/N$. Clearly, $u = l + N = e + N$ for some $l \in L \setminus N$ and $e \in E \setminus N$. But this shows that $l \in E$. Hence, we get that $E$ is proper essential in $M$. 

If $N = \{0_{M}\}$, there is nothing to show. Therefore, assume $N \neq \{0_{M}\}$. Let $L/N$ be a non-trivial proper submodule of $M/N$. Let $l \in L \setminus N$ be such $rl \in E$, for some nonzero $r \in R$ and $rl \neq 0_{M}$. As $rl \notin N$, we conclude that $rl + N \in (E + N)/N \cap L/N$, that is, $(E + N)/N$ is proper essential in $M/N$.
\end{proof}

\begin{corollary}
Let $M$ be a left module over an integral domain $R$ that contains $D^{-1}R$ as its submodule and $M \neq \mathrm{Tor}(M)$. Then, the $R$-module $M$ has a proper essential submodule.
\end{corollary}
\begin{proof}
Consider the quotient module $M/\mathrm{Tor}(M)$. As $M/\mathrm{Tor}(M)$ is torsion-free, so applying theorem \eqref{jpmkj} we get that $M/\mathrm{Tor}(M)$ has a proper essential submodule. Again, applying lemma \eqref{gmkj}, we see that $M$ has a proper essential submodule.
\end{proof}
The correspondence between elements in $\mathrm{Hom}_{R}(M, N)$ with matrices over $R$ plays a fundamental role in the theory of modules. Keeping this in mind, the most natural question that we investigate is whether the set of matrices (of some dimension) over $R$ (as an $R$-module) has a proper essential submodule. We obtain an affirmative answer in the following theorem. Let $\mathcal{M}_{m \times n}(R)$ denote the set of all $m \times n$ matrices over $R$.

\begin{theorem}
Let $R$ be an integral domian. Then the left $R$-module $\mathcal{M}_{m \times n}(R)$ possess a proper essential submodule.
\end{theorem}
\begin{proof}
Let $M$ and $N$ are the free $R$ modules on the sets of $n$ and $m$ elements respectively. Clearly, we have $M \simeq R^{n}$ and $N \simeq R^{m}$. As $\mathrm{Hom}_{R}(R^{n}, R^{m}) \simeq \mathrm{Hom}_{R}(M, N)$ and $\mathrm{Hom}_{R}(M, N) \simeq \mathcal{M}_{m \times n}(R)$, we get that $\mathrm{Hom}_{R}(R^{n}, R^{m}) \simeq \mathcal{M}_{m \times n}(R)$. Since we have $\mathrm{Hom}_{R}(R^{n}, R^{m}) \simeq \bigoplus_{i = 1}^{n} \bigoplus_{j = 1}^{m} \mathrm{Hom}_{R}(R, R)$, we obtain $\mathcal{M}_{m \times n}(R) \simeq R^{mn}$. Now, we can apply theorem \eqref{spm} to obtain the desired result.
\end{proof}

\section{Essential socles in a new class of modules}

\begin{definition}
A left $R$-module $M$ is said to be an $\mathrm{SM}$ module if every nontrivial submodule of $M$ contains a simple submodule.
\end{definition}

\begin{proposition}
Let $M$ be a left module over $R$ such that $M$ is either Artinian, or, locally finite, or, is of finite length. Then $M$ is an $\mathrm{SM}$ module.
\end{proposition}
\begin{proof}
Let $M$ be an Artinian module. If $M$ is simple there is nothing to show. Therefore assume $M$ is not simple. Let $N$ be any submodule of $M$. If we consider any strictly descending chain $C_{N}: N_{0} = N \supsetneq N_{1} \supseteq N_{2} \supsetneq \cdots \supsetneq N_{s}$ of submodules of $N$ that terminates with the nontrivial proper submodule $N_{s}$, then $N_{s}$ is a simple module contained in $N$. Because, otherwise one can enlarge the chain $C_{N}$.

Now assume $M$ is a locally finite module. Again, if $M$ is simple then $M$ must be finite and there is nothing to show. So, assume $M$ is not simple. Let $N$ be any submodule of $M$ and let $n \in N$. Let $\mathcal{G}$ be the collection of all submodules of $\langle n \rangle$. As $\langle n \rangle$ is finite, we can choose $S \in \mathcal{G}$ such that $|S|$ is least. Clearly, $S$ can not contain any nontrivial proper submodule $L$, because then $L$ being a member of $\mathcal{G}$ contradicts the minimality of $|S|$. This shows that $S$ is a simple submodule contained in $N$.

Finally, assume that the length of $M$ is finite. If $l(M) = 1$, then $M$ is simple and there is nothing to show. So, assume $l(M) > 1$. Let $N$ be any submodule of $M$. Let $\mathcal{H}$ be the collection of all nontrivial proper submodule of $N$. Choose $K \in \mathcal{H}$ such that $l(K)$ is minimal. If $K$ would not be simple and $P$ is a nontrivial proper submodule of $K$, then $l(P) < l(K)$ contradicts the minimality of $l(K)$. Thus, $K$ is a simple submodule contained in $N$.
\end{proof}

\begin{theorem}
The class of left $\mathrm{SM}$ modules properly contains the class of Artinian modules, locally finite modules, and modules of finite lengths.
\end{theorem}
\begin{proof}
Let $\mathcal{C}_{1}, \mathcal{C}_{2}, \mathcal{C}_{3}$, and $\mathcal{C}_{4}$ denote the class of all $\mathrm{SM}$ modules, modules of finite lengths, locally finite modules, and Artinian modules respectively. Consider the free $\mathbb{Z}_{p^{n}}$-module $M = \bigoplus_{\mathbb{N}} \mathbb{Z}_{p^{n}}$ on the set $\mathbb{N}$, where $p$ is some fixed prime and $n$ is a natural number such that $n > 1$. Let $N$ be any nontrivial proper submodule of $M$ and let $n \in N$. Observe that $|\langle n \rangle| \leq p^{n}$. If we choose a submodule $L$ of $\langle n \rangle$ such that $|L| = p$, then clearly $L$ is simple and contained in $N$. This shows that $M$ is an $\mathrm{SM}$ module.

As $l(\mathbb{Z}_{p^{n}}) = n$, therefore the divergence of the series $\sum_{i = 1}^{\infty} n$  shows that the length of the free $\mathbb{Z}_{p^{n}}$-module $\bigoplus_{\mathbb{N}} \mathbb{Z}_{p^{n}}$ is not finite. Hence, we have that $\mathcal{C}_{2} \subsetneq \mathcal{C}_{1}$.

For each natural number $i$, we define the submodules $L_{i} = \bigoplus_{\mathbb{N}} F_{j}$, where $F_{j} = \mathbb{Z}_{p^{n}}$ if $j > i$, otherwise, $F_{j} = \{0_{\mathbb{Z}_{p^{n}}}\}$. Now, observe that the strictly descending chain of submodules $L_{0} = M \supsetneq L_{1} \supsetneq L_{2} \supsetneq \cdots L_{s} \supsetneq \cdots$ do not terminate. Thus, $\mathcal{C}_{4} \subsetneq \mathcal{C}_{1}$.

Now, consider an infinite $R$-module $W$ of finite length $n$ (for example finite dimensional vector space). If $W_{0} = W \supsetneq W_{1} \supsetneq W_{2} \supseteq \cdots \supsetneq W_{n} = \{0_{W}\}$, then observe that $W = R(A)$, where $A = \{w_{i} \mid w_{i} \in W_{i} \setminus W_{i + 1}, 0 \leq i \leq n - 2\}$ and $R(A)$ is the module generated by the subset $A$. As $W \notin \mathcal{C}_{3}$, therefore, we get that $\mathcal{C}_{3} \subsetneq \mathcal{C}_{1}$.
\end{proof}

\begin{theorem}
Let $M$ be a left $\mathrm{SM}$ module over $R$. Then $\mathrm{Soc}(M)$ is proper essential in $M$ if and only if $M$ is not semisimple.
\end{theorem}
\begin{proof} 
Assume first that $\mathrm{Soc}(M)$ is proper essential in $M$. Then, clearly we have $\mathrm{Soc}(M) \subsetneq M$. Hence, applying theorem \eqref{babama}, we conclude that $M$ is not semisimple.

Conversely, assume that $M$ is not semisimple. Again, applying theorem \eqref{babama}, we see that $M$ contains a proper essential submodule. Now, let $N$ be any submodule of $M$. As $N$ contains a simple submodule, we get that $\mathrm{Soc}(M) \cap N \neq \{0_{M}\}$. Further, as $\mathrm{Soc}(M)$ is contained in every proper essential submodule of $M$, we conclude that $\mathrm{Soc}(M)$ is proper essential in $M$.  
\end{proof}

\section{Torsion-free modules over integral domains}

In general, it is not clear from theorem \eqref{jpmkj} that under what condition a left torsion-free module $M$ over an integral domain $R$ contains an essential submodule. For this purpose, we define the closure of a submodule $N$ of $M$ that gives some sort of measure in the sense that it is that maximal subspace of $M$ in which $N$ is essential. 

\begin{definition}\label{madurga}
Let $M$ be a left torsion-free module over an integral domain $R$. For any submodule $N$ of $M$, we denote $\mathrm{Cl}_{M}(N)$ to mean the closure of $N$ in $M$ defined in the following way: $\mathrm{Cl}_{M}(N) = \{m \in M \mid rm \in N, \text{for some nonzero $r \in R$}\}$.  
\end{definition}

Throughout this section, unless otherwise stated, we assume $M$ is a left torsion-free module over an integral domain $R$ and $D$ is the set of nonzero elements in $R$. The following proposition is quite straightforward and follows directly from definition \eqref{madurga}.  

\begin{proposition}\label{debanandamaharaj}
Let $N, N_{1}, N_{2}, T$ are submodules of $M$ and let $f: M \rightarrow S$ be any homomorphism of modules. Then, we have the following:

$\mathrm{(a)}$ $\mathrm{Cl}_{M}(\{0_{M}\}) = \{0_{M}\}$ and $\mathrm{Cl}_{M}(N)$ is a submodule of $M$.

$\mathrm{(b)}$ $N \subseteq \mathrm{Cl}_{M}(N)$ and $\mathrm{Cl}_{M}(\mathrm{Cl}_{M}(N)) = \mathrm{Cl}_{M}(N)$.

$\mathrm{(c)}$ If $N_{1} \subseteq N_{2}$, then $\mathrm{Cl}_{M}(N_{1}) \subseteq \mathrm{Cl}_{M}(N_{2})$.

$\mathrm{(d)}$ $f(\mathrm{Cl}_{M}(T)) \leq \mathrm{Cl}_{S}(f(T))$.
\end{proposition}
\begin{proof}
$\mathrm{(a)}:$ $\mathrm{Cl}_{M}(\{0_{M}\})$ is the set of all those elements $m$ in $M$ such that $rm = 0$ for some nonzero $r \in R$. As $M$ is torsion-free, we see that $\mathrm{Cl}_{M}(\{0_{M}\}) = \{0_{M}\}$. If $m_{1}, m_{2}$ are in $\mathrm{Cl}_{M}(N)$, then $r_{1}m_{1} \in N$ and $r_{2}m_{2} \in N$ for some nonzero elements $r_{1}, r_{2}$ in $R$. Since $r_{1}r_{2}(m_{1} - m_{2}) \in N$ and $r_{1}r_{2} \neq 0$, we conclude that $(m_{1} - m_{2}) \in \mathrm{Cl}_{M}(N)$. Now, let $m \in \mathrm{Cl}_{M}(N)$ and $r \in R$ be such that $r \neq 0$ and $rm \in N$. For any $s \in R$, we have $r(sm) \in N$. Thus, $sm \in \mathrm{Cl}_{M}(N)$. Hence, $\mathrm{Cl}_{M}(N)$ is a submodule of $M$.

$\mathrm{(b)}:$ For any element $n \in N$, we have $1n \in N$, so $N \subseteq \mathrm{Cl}_{M}(N)$. It is evident that $\mathrm{Cl}_{M}(N) \subseteq \mathrm{Cl}_{M}(\mathrm{Cl}_{M}(N))$. Now, suppose $m \in \mathrm{Cl}_{M}(\mathrm{Cl}_{M}(N))$. This implies, for some nonzero $r \in R$, $rm$ lie in $\mathrm{Cl}_{M}(N)$. Again, for some nonzero $s \in R$, $(sr)m \in N$. As, $sr \neq 0$, we conclude that $m \in \mathrm{Cl}_{M}(N)$. Hence, $\mathrm{Cl}_{M}(\mathrm{Cl}_{M}(N)) \subseteq \mathrm{Cl}_{M}(N)$.

$\mathrm{(c)}:$ Let $m \in \mathrm{Cl}_{M}(N_{1})$. Then, $rm \in N_{1}$ for some nonzero $r \in R$. As $rm \in N_{2}$ as well, we conclude that $m \in \mathrm{Cl}_{M}(N_{2})$, that is, $\mathrm{Cl}_{M}(N_{1}) \subseteq \mathrm{Cl}_{M}(N_{2})$.

$\mathrm{(d)}:$ Let $s \in f(\mathrm{Cl}_{M}(T))$. Then, there exists $m \in \mathrm{Cl}_{M}(T)$ such that $f(m) = s$. Also, there exists $r \in R$ such that $r \neq 0$ and $rm \in T$. This shows that $rf(m) \in f(T)$, that is, $s \in \mathrm{Cl}_{S}(f(T))$. Hence, we get that $f(\mathrm{Cl}_{M}(T)) \leq \mathrm{Cl}_{S}(f(T))$.
\end{proof}

The question that we investigate is whether the closure operator commutes with arbitrary direct sums of submodules. We see that it is not true in general, but in bigger space, the closure operator commutes with arbitrary direct sum of submodules. 

\begin{lemma}\label{joymadurga}
Let $N_{1}, N_{2}, \ldots, N_{n}$ are the submodules of $M$. If $D$ denotes the set of all nonzero elements in $R$, then $\mathrm{Cl}_{D^{-1}M}\big(\bigoplus_{i = 1}^{n} N_{i}\big) = \bigoplus_{i = 1}^{n} \mathrm{Cl}_{D^{-1}M}(N_{i})$. 
\end{lemma}
\begin{proof}
It is enough to show that $\mathrm{Cl}_{D^{-1}M}(L \bigoplus S) = \mathrm{Cl}_{D^{-1}M}(L) \bigoplus \mathrm{Cl}_{D^{-1}M}(S)$ for any two submodules $L, S$ of $M$. Observe that if $\frac{m}{d} \in \mathrm{Cl}_{D^{-1}M}(L) \cap \mathrm{Cl}_{D^{-1}M}(S)$, then $\frac{r_{1}m}{d} \in L$ and $\frac{r_{2}m}{d} \in S$ for some nonzero $r_{1}, r_{2} \in R$. As $(r_{1}r_{2})\frac{m}{d} \in L \cap S$, we see that $\frac{m}{d} = 0_{M}$. Now, suppose $(\frac{m_{1}}{d_{1}} + \frac{m_{2}}{d_{2}}) \in \mathrm{Cl}_{D^{-1}M}(L) \bigoplus \mathrm{Cl}_{D^{-1}M}(S)$. We have then $\frac{r_{1}m_{1}}{d_{1}} \in L$ and $\frac{r_{2}m_{2}}{d_{2}} \in S$ for some $r_{1}, r_{2} \in R \setminus \{0\}$. As $(r_{1}r_{2})(\frac{m_{1}}{d_{1}} + \frac{m_{2}}{d_{2}}) \in L \bigoplus S$, we see that $(\frac{m_{1}}{d_{1}} + \frac{m_{2}}{d_{2}}) \in \mathrm{Cl}_{D^{-1}M}(L \bigoplus S)$. Conversely, let $\frac{m}{d} \in \mathrm{Cl}_{D^{-1}M}(L \bigoplus S)$. Therefore, for some nonzero $r \in R$, we have $\frac{rm}{d} \in L \bigoplus S$. Let $\frac{rm}{d} = l + s$. Then, $r(\frac{m}{d} - \frac{l}{r}) = s$, that is, $(\frac{m}{d} - \frac{l}{r}) \in \mathrm{Cl}_{D^{-1}M}(S)$. Also, observe that $\frac{l}{r} \in \mathrm{Cl}_{D^{-1}M}(L)$. Hence, we get that $\frac{m}{d} = \frac{l}{r} + (\frac{m}{d} - \frac{l}{r})$, that is, $\frac{m}{d} \in \mathrm{Cl}_{D^{-1}M}(L) \bigoplus \mathrm{Cl}_{D^{-1}M}(S)$.
\end{proof}

\begin{theorem}\label{joymakamakhyama}
If $\{N_{\alpha}\}_{\alpha \in \Lambda}$ be a family of submodules indexed by a nonempty set $\Lambda$, then $\mathrm{Cl}_{D^{-1}M}\big(\bigoplus_{\alpha \in \Lambda} N_{\alpha}\big) = \bigoplus_{\alpha \in \Lambda} \mathrm{Cl}_{D^{-1}M}(N_{\alpha})$. In particular, we have the following, $\mathrm{Cl}_{M}\big(\bigoplus_{\alpha \in \Lambda} N_{\alpha}\big) = M \cap \big(D^{-1}\big(\bigoplus_{\alpha \in \Lambda} \mathrm{Cl}_{M}(N_{\alpha})\big)\big)$.
\end{theorem}
\begin{proof}
Let $\frac{m}{d} \in \mathrm{Cl}_{D^{-1}M}\big(\bigoplus_{\alpha \in \Lambda} N_{\alpha}\big)$. Then, for some $r \neq 0$, we have $\frac{rm}{d} \in \bigoplus_{\alpha \in \Lambda} N_{\alpha}$. Let $\frac{rm}{d} = \sum_{i = 1}^{n} n_{\alpha_{i}}$. This implies, $\frac{m}{d} \in \mathrm{Cl}_{D^{-1}M}(\bigoplus_{i = 1}^{n} N_{\alpha_{i}})$. Now, applying lemma \eqref{joymadurga}, we get that $\mathrm{Cl}_{D^{-1}M}(\bigoplus_{i = 1}^{n} N_{\alpha_{i}}) \leq \bigoplus_{\alpha \in \Lambda} \mathrm{Cl}_{D^{-1}M}(N_{\alpha})$. Hence, we obtain that $\frac{m}{d} \in \bigoplus_{\alpha \in \Lambda} \mathrm{Cl}_{D^{-1}M}(N_{\alpha})$. Conversely, suppose $\sum_{i = 1}^{n} \frac{m_{\alpha_{i}}}{d_{\alpha_{i}}} \in \bigoplus_{\alpha \in \Lambda} \mathrm{Cl}_{D^{-1}M}(N_{\alpha})$, where $\frac{m_{\alpha_{i}}}{d_{\alpha_{i}}} \in \mathrm{Cl}_{D^{-1}M}(N_{\alpha_{i}})$. Then, for each $i$ with $1 \leq i \leq n$, there exists some nonzero $r_{\alpha_{i}}$ such that $\frac{r_{\alpha_{i}}m_{\alpha_{i}}}{d_{\alpha_{i}}} \in N_{\alpha_{i}}$. But then $(r_{\alpha_{1}}r_{\alpha_{2}} \cdots r_{\alpha_{n}})\big(\sum_{i = 1}^{n} \frac{m_{\alpha_{i}}}{d_{\alpha_{i}}}\big) \in \bigoplus_{i = 1}^{n} N_{\alpha_{i}}$ shows that $(r_{\alpha_{1}}r_{\alpha_{2}} \cdots r_{\alpha_{n}})\big(\sum_{i = 1}^{n} \frac{m_{\alpha_{i}}}{d_{\alpha_{i}}}\big) \in \bigoplus_{\alpha \in \Lambda} N_{\alpha}$, that is, $\sum_{i = 1}^{n} \frac{m_{\alpha_{i}}}{d_{\alpha_{i}}} \in \mathrm{Cl}_{D^{-1}M}\big(\bigoplus_{\alpha \in \Lambda} N_{\alpha}\big)$.

For any submodule $N$ of $M$, if we show that $\mathrm{Cl}_{D^{-1}M}(N) = D^{-1}\mathrm{Cl}_{M}(N)$, then we have $\bigoplus_{\alpha \in \Lambda} \mathrm{Cl}_{D^{-1}M}(N_{\alpha}) = D^{-1}\big(\bigoplus_{\alpha \in \Lambda} \mathrm{Cl}_{M}(N_{\alpha})\big)$, since $D^{-1}$ distributes over arbitrary direct sum. So, suppose $\frac{m}{d} \in D^{-1}\mathrm{Cl}_{M}(N)$. As $m \in \mathrm{Cl}_{M}(N)$, we have $rm \in N$ for some nonzero $r \in R$. But then, we have $(rd)(\frac{m}{d}) \in N$. Hence, $\frac{m}{d} \in \mathrm{Cl}_{D^{-1}M}(N)$. Conversely, suppose $\frac{m}{d} \in \mathrm{Cl}_{D^{-1}M}(N)$. Then, $\frac{rm}{d} \in N$ for some nonzero $r \in R$. This implies, $rm \in N$, that is, $m \in \mathrm{Cl}_{M}(N)$. Hence, we get that $\frac{m}{d} \in D^{-1}\mathrm{Cl}_{M}(N)$. The proof is now completed by taking intersection on both sides of the equation $\mathrm{Cl}_{D^{-1}M}\big(\bigoplus_{\alpha \in \Lambda} N_{\alpha}\big) = D^{-1}\big(\bigoplus_{\alpha \in \Lambda} \mathrm{Cl}_{M}(N_{\alpha})\big)$ with $M$.
\end{proof}

\begin{theorem}\label{adisoktiaddyasokti}
If $N$ is a submodule of $M$, then $\mathrm{Soc}\big(\mathrm{Cl}_{M}(N)\big) = \mathrm{Cl}_{M}\big(\mathrm{Soc}(N)\big)$. In particular, they are equal to $\mathrm{Soc}(N)$.
\end{theorem}
\begin{proof}
If $S$ is a simple submodule of $N$ then clearly it is also simle in $\mathrm{Cl}_{M}(N)$. Conversely, if $S = \langle t \rangle$ is a simple submodule of $\mathrm{Cl}_{M}(N)$, then $rt \in N$ for some nonzero $r \in R$. As $\langle rt \rangle = \langle t \rangle$, we conclude that $S$ is a simple submodule of $N$. Hence, we conclude that $\mathrm{Soc}\big(\mathrm{Cl}_{M}(N)\big) = \mathrm{Soc}(N)$. Let $\mathrm{Soc}(N) = \bigoplus_{\alpha \in \Lambda} S_{\alpha}$, where $S_{\alpha}$ is a simple submodule of $N$ for each $\alpha \in \Lambda$. Now, using theorem \eqref{joymakamakhyama}, we get that $\mathrm{Cl}_{D^{-1}M}\big(\mathrm{Soc}(N)\big) = \bigoplus_{\alpha \in \Lambda} \mathrm{Cl}_{D^{-1}M}(S_{\alpha})$. As simple submodules of $M$ are simple submodules of $D^{-1}M$ and vice-versa, we see that $S_{\alpha}$ is a simple submodule of $D^{-1}M$ for each $\alpha \in \Lambda$. In this position, we wish to use the fact that for any left torsion-free module $Z$ (over some integral domain $B$) and for any simple submodule $L$ of $Z$, we have $\mathrm{Cl}_{Z}(L) = L$. To see this, let $z \in \mathrm{Cl}_{Z}(L)$ and $bz \in L$ for some nonzero $b \in B$. Let $L = \langle l \rangle$. Clearly, we have $\langle l \rangle = \langle bl \rangle = \langle bz \rangle$. If $bz = bl$, then $b(z - l) = 0$. As $Z$ is torsion-free and $b \neq 0$, we conclude that $z = l$, that is $z \in L$. If $bz \neq bl$, then for some nonzero $b^{\prime} \in B$, we have $bz = (b^{\prime}b)l$, that is, $b(z - b^{\prime}l) = 0$. Similar arguments as above shows that $z = b^{\prime}l$, that is, $z \in L$. Hence, $\mathrm{Cl}_{Z}(L) \subseteq L$. Now, applying proposition \eqref{debanandamaharaj}, we conclude that $\mathrm{Cl}_{Z}(L) = L$. Thus, we have that $\mathrm{Cl}_{D^{-1}M}\big(\mathrm{Soc}(N)\big) = \bigoplus_{\alpha \in \Lambda} S_{\alpha}$, that is, $\mathrm{Cl}_{D^{-1}M}\big(\mathrm{Soc}(N)\big) = \mathrm{Soc}(N)$. Now, $M \cap \mathrm{Cl}_{D^{-1}M}\big(\mathrm{Soc}(N)\big) = \mathrm{Cl}_{M}\big(\mathrm{Soc}(N)\big)$ shows that $\mathrm{Cl}_{M}\big(\mathrm{Soc}(N)\big) = \mathrm{Soc}(N)$. Hence, $\mathrm{Soc}\big(\mathrm{Cl}_{M}(N)\big) = \mathrm{Cl}_{M}\big(\mathrm{Soc}(N)\big) = \mathrm{Soc}(N)$.
\end{proof}

\begin{corollary}\label{amabjg}
Left torsion-free modules over integral domains are either semisimple or do not possess essential socles.
\end{corollary}
\begin{proof}
If $M$ is not semisimple then $\mathrm{Soc}(M) \neq M$ (see theorem \eqref{baba}). Now, if we set $N = M$ in theorem \eqref{adisoktiaddyasokti}, we see that $\mathrm{Cl}_{M}\big(\mathrm{Soc}(M)\big) = \mathrm{Soc}(M)$. Thus, the socle $\mathrm{Soc}(M)$ is not essential in $M$.
\end{proof}

\begin{theorem}\label{joybabamahadev}
For submodules $U, V$ of $M$, we have the following exact sequence: $0 \rightarrow D^{-1}\mathrm{Cl}_{M}(U) \xrightarrow{\alpha^{\prime}} D^{-1}\mathrm{Cl}_{M}(V) \xrightarrow{\beta^{\prime}} D^{-1}\mathrm{Cl}_{M/U}(V/U) \rightarrow 0$. Moreover, it splits if the sequence $0 \rightarrow U \xrightarrow{\alpha} V \xrightarrow{\beta} V/U \rightarrow 0$ splits.
\end{theorem}
\begin{proof}
Let 
$\varphi: \mathrm{Cl}_{D^{-1}M/D^{-1}U}(D^{-1}V/D^{-1}U) \rightarrow \mathrm{Cl}_{D^{-1}M}(D^{-1}V)/\mathrm{Cl}_{D^{-1}M}(D^{-1}U)$
be a map defined by $\varphi: \frac{m}{d} + D^{-1}U \mapsto \frac{m}{d} + \mathrm{Cl}_{D^{-1}M}(D^{-1}U)$. Observe that we have $\frac{m}{d} + \mathrm{Cl}_{D^{-1}M}(D^{-1}U) \in \mathrm{Cl}_{D^{-1}M}(D^{-1}V)/\mathrm{Cl}_{D^{-1}M}(D^{-1}U)$. Because, for some nonzero $r \in R$, we have $\frac{rm}{d} + D^{-1}U = \frac{v}{d^{\prime}} + D^{-1}U$. But this implies $r(\frac{m}{d} - \frac{v}{rd^{\prime}}) \in D^{-1}U$, that is, $(\frac{m}{d} - \frac{v}{rd^{\prime}}) \in \mathrm{Cl}_{D^{-1}M}(D^{-1}U)$. As $\frac{v}{rd^{\prime}} \in \mathrm{Cl}_{D^{-1}M}(D^{-1}V)$, $\frac{m}{d} + \mathrm{Cl}_{D^{-1}M}(D^{-1}U)$ is an element of $\mathrm{Cl}_{D^{-1}M}(D^{-1}V)/\mathrm{Cl}_{D^{-1}M}(D^{-1}U)$. To see that $\varphi$ is well-defined, observe that for elements  $\frac{m}{d} + D^{-1}U, \frac{m^{\prime}}{d^{\prime}} + D^{-1}U$ in $\mathrm{Cl}_{D^{-1}M/D^{-1}U}(D^{-1}V/D^{-1}U)$, if we have $(\frac{m}{d} - \frac{m^{\prime}}{d^{\prime}}) \in D^{-1}U$, then $\frac{m}{d} + \mathrm{Cl}_{D^{-1}M}(D^{-1}U) = \frac{m^{\prime}}{d^{\prime}} + \mathrm{Cl}_{D^{-1}M}(D^{-1}U)$, since proposition \eqref{debanandamaharaj} yields $D^{-1}U \subseteq \mathrm{Cl}_{D^{-1}M}(D^{-1}U)$. It is easy to see that $\varphi$ is a module homomorphism. If $\frac{m}{d} + D^{-1}U \in \mathrm{Cl}_{D^{-1}M/D^{-1}U}(D^{-1}V/D^{-1}U)$ be such that $\frac{m}{d} \in \mathrm{Cl}_{D^{-1}M}(D^{-1}U)$, then $\frac{rm}{d} \in D^{-1}U$ for some nonzero $r \in R$, that is, $\frac{m}{d} \in D^{-1}U$. This shows that $\varphi$ is injective. Finally, for any $\frac{m}{d} + \mathrm{Cl}_{D^{-1}M}(D^{-1}U) \in \mathrm{Cl}_{D^{-1}M}(D^{-1}V)/\mathrm{Cl}_{D^{-1}M}(D^{-1}U)$, we have $\frac{rm}{d} \in D^{-1}V$ for some nonzero $r \in R$. As $r(\frac{m}{d} + D^{-1}U) \in D^{-1}V/D^{-1}U$, we conclude that $(\frac{m}{d} + D^{-1}U) \in \mathrm{Cl}_{D^{-1}M/D^{-1}U}(D^{-1}V/D^{-1}U)$, that is, $\varphi$ is surjective. Hence, $\varphi$ is an isomorphism. It is now easy to see that $\mathrm{Cl}_{D^{-1}M/D^{-1}U}(D^{-1}V/D^{-1}U) = D^{-1}\mathrm{Cl}_{M/U}(V/U)$ and $\mathrm{Cl}_{D^{-1}M}(D^{-1}V)/\mathrm{Cl}_{D^{-1}M}(D^{-1}U) = D^{-1}(\mathrm{Cl}_{M}(V)/\mathrm{Cl}_{M}(U))$. This completes the first part of the theorem. Applying lemma \eqref{joymadurga}, we can conclude the second part.
\end{proof}

We close this section with the following theorem that can also be viewed as an application of corollary \eqref{amabjg}.

\begin{theorem}
Left modules over integral domains with nonempty torsion-free parts are either semisimple or do not possess essential socles.
\end{theorem}
\begin{proof}
If $M$ is semisimple, there is nothing to show. So, assume that $M$ is not semisimple. This implies $M$ has a proper essential submodule $E$ (see theorem \eqref{baba}). Now, applying lemma \eqref{gmkj}, we see that $(E + \mathrm{Tor}(M))/\mathrm{Tor}(M)$ is proper essential in $M/\mathrm{Tor}(M)$. Now, if $\mathcal{P}$ be the collection of all essential submodules of $M/\mathrm{Tor}(M)$ and if we denote $\overline{D} = D/\mathrm{Tor}(M)$ for $D/\mathrm{Tor}(M) \in \mathcal{P}$, then from corollary \eqref{amabjg}, we can conclude that the submodule $\bigcap_{\overline{D} \in \mathcal{P}} \overline{D}$ is not proper essential in $M/\mathrm{Tor}(M)$. Since we have $\big(\bigcap_{\overline{D} \in \mathcal{P}} D\big)/\mathrm{Tor}(M) \leq \bigcap_{\overline{D} \in \mathcal{P}} \overline{D}$, therefore, we can conclude from lemma \eqref{gmkj} that the submodule $\bigcap_{\overline{D} \in \mathcal{P}} D$ can not be proper essential in $M$. Further, since $\mathrm{Soc}(M)$ is a submodule of $\bigcap_{\overline{D} \in \mathcal{P}} D$, we conclude that $M$ can not possess an essential socle.
\end{proof}

\end{document}